\documentclass{svproc}

 \usepackage[utf8]{inputenc}
 \usepackage[T1]{fontenc}
 \usepackage{lmodern}
 \usepackage{amsmath}
 \usepackage{amssymb}
 \usepackage{hyperref}

 \renewcommand{\d}{\mathrm d}

\begin{document}
\mainmatter

\title{Long-time dynamics for a simple aggregation equation on the sphere}

\author{Amic Frouvelle\inst{1} \and Jian-Guo Liu\inst{2}}
\institute{CEREMADE, CNRS, Université Paris-Dauphine, Université PSL, 75016 Paris, France, \email{frouvelle@ceremade.dauphine.fr} \and Department of Physics and Department of Mathematics, Duke University, Durham, NC 27708, USA, \email{jliu@phy.duke.edu}}
\maketitle

\begin{abstract}
We give a complete study of the asymptotic behavior of a simple model of alignment of unit vectors, both at the level of particles, which corresponds to a system of coupled differential equations, and at the continuum level, under the form of an aggregation equation on the sphere. We prove unconditional convergence towards an aligned asymptotic state. In the cases of the differential system and of symmetric initial data for the partial differential equation, we provide precise rates of convergence.
\keywords{alignment, unit vectors, aggregation equation}
\end{abstract}

\section{Introduction and main results}

We are interested in a model of alignment of unit vectors. Our interest comes from the mechanism of alignment of self-propelled particles presented by Degond and Motsch in~\cite{degond2008continuum}, which is a time-continuous model inspired from the Vicsek model~\cite{vicsek1995novel} (in which the alignment process is discrete in time). In these models, the velocities of the particles, considered as unit vectors, try to align towards the average orientation of their neighbors and are subject to some angular noise. We want to study the simple case without spatial dependence and without noise. More precisely, at the level of the particle dynamics, we consider the deterministic part of the spatially homogeneous model of~\cite{bolley2012meanfield}, which corresponds to a regularized version of~\cite{degond2008continuum}: the particles align with the average velocity of the others (instead of dividing this average vector by its norm to get a averaged orientation). It reads as
\begin{equation}
  \label{ode-intro}\frac{\d v_i}{\d t}=P_{v_i^\perp}J, \quad\text{ with } J=\frac1{N}\sum_{j=1}^Nv_j,
\end{equation}
where~$(v_i)_{1\leqslant i\leqslant N}$ are~$N$ unit vectors belonging to~$\mathbb{S}$, the unit sphere of~$\mathbb{R}^n$, and~$P_{v^\perp}$ is the projection on the orthogonal of a unit vector~$v\in\mathbb{S}$, given by~$P_{v^\perp}u=u-(v\cdot u) v$ for~$u\in\mathbb{R}^n$. This projection ensures that the velocities stay of norm one for all positive times. This system of equations can be seen as alignment towards the unit vector pointing in the same direction as~$J$ (the average of all velocities). Indeed the term~$P_{v^\perp}J$ is equal to~$\nabla_v(J\cdot v)$, where~$\nabla_v$ is the gradient operator on the unit sphere~$\mathbb{S}$. Therefore the dynamics of a particle following the equation~$\frac{\d v}{\d t}=\nabla_v(v\cdot J)$ corresponds to the maximization of this quantity~$v\cdot J$, which is maximal when~$v$ is aligned in the same direction as~$J$.

At the kinetic level, we are interested in the evolution of a probability measure~$f(t,\cdot)$ on~$\mathbb{S}$ given by
\begin{equation}
  \label{pde-intro}
  \partial_tf+\nabla_v\cdot(fP_{v^\perp}J_f)=0,\quad\text{ with } J_f=\int_\mathbb{S}v f \d v,
\end{equation}
where~$\nabla_v\cdot$ is the divergence operator on the sphere~$\mathbb{S}$. The link between this evolution equation and the system of ordinary differential equations~\eqref{ode-intro}, is that if the measure~$f$ is the so-called empirical distribution of the particles~$(v_i)_{1\leqslant i\leqslant N}$, given by~$f=\frac1{N}\sum_{i=1}^N\delta_{v_i}$, then it is a weak solution of the kinetic equation~\eqref{pde-intro} if and only if the vectors~$(v_i)_{1\leqslant i\leqslant N}$ are solutions of the system~\eqref{ode-intro} (see Remark~\ref{remark-ode-pde}). This kinetic equation~\eqref{pde-intro} corresponds to the spatially homogeneous version of the mean-field limit of~\cite{bolley2012meanfield} in which the diffusion coefficient has been set to zero. The case with a positive diffusion has been treated in detail in~\cite{frouvelle2012dynamics} by the authors of the present paper, and it presents a phenomenon of phase transition: when the diffusion coefficient is greater than a precise threshold, all the solutions converge exponentially fast towards the uniform measure on the sphere~$\mathbb{S}$, and when it is smaller, all solutions except those for which~$J_f$ is initially zero converge exponentially fast to a non-isotropic steady-state (a von Mises distribution). When the diffusion coefficient tends to zero, the von Mises distributions converge to Dirac measures concentrated at one point of~$\mathbb{S}$. Therefore, we can expect that the solutions of~\eqref{pde-intro} converge to a Dirac measure. The main object of this paper is to make this statement precise, in proving the following theorem:

\begin{theorem}\label{thm-pde}
Let~$f_0$ be a probability measure on~$\mathbb{S}$ of~$\mathbb{R}^n$, and~$f\in C(\mathbb{R}_+,\mathcal{P}(\mathbb{S}))$ be the solution of~\eqref{pde-intro} with initial condition~$f(0,v)=f_0(v)$.

If~$J_f(0)\neq0$, then~$t\mapsto|J_f(t)|$ is nondecreasing, so~$\Omega(t)=\frac{J_f(t)}{|J_f(t)|}\in\mathbb{S}$ is well-defined for all times~$t\geqslant0$. Furthermore there exists~$\Omega_\infty\in\mathbb{S}$ such that~$\Omega(t)$ converges to~$\Omega_\infty(t)$ as~$t\to+\infty$.

Finally, there exists a unique~$v_{\mathrm{back}}\in\mathbb{S}$ such that the solution of the differential equation~$\frac{\d v}{\d t}=P_{v^\perp}J_f(t)$ with initial condition~$v(0)=v_{\mathrm{back}}$ is such that~$v(t)\to-\Omega_\infty$ as~$t\to\infty$. Then, if we denote by~$m$ the mass of the singleton~$\{v_{\mathrm{back}}\}$ with respect to the measure~$f_0$, we have~$m<\frac12$ and~$f(t,\cdot)$ converges weakly as~$t\to\infty$ towards the measure~$(1-m)\delta_{\Omega_\infty}+m\delta_{-\Omega_\infty}$.
\end{theorem}

In particular, this theorem shows that if the initial condition~$f_0$ has no atoms and satisfies~$J_{f_0}\neq0$, then the measure~$f$ converges weakly to a Dirac mass at some~$\Omega_\infty\in\mathbb{S}$. Let us mention that there is no rate of convergence in this theorem. In general, there is no hope to have such a rate for an arbitrary initial condition (see Proposition~\ref{prop-no-rate}), but under regularity assumptions, one can expect to have an exponential rate of convergence (this is the case when the initial condition has some symmetries implying that~$\Omega(t)$ is constant, see Proposition~\ref{prop-rates-regular}).

We will also study in detail the system of ordinary differential equations~\eqref{ode-intro}. Since this is a particular case of~\eqref{pde-intro} in the case where~$f=\frac1{N}\sum_{i=1}^N\delta_{v_i}$ (see Remark~\ref{remark-ode-pde}), we can apply the main theorem, but now the measure~$f$ has atoms, and actually we will see that working directly with the differential equations allows to have more precise results such as exponential rates of convergence. For instance the quantity~$\Omega(t)$ plays the role as a nearly conserved quantity, as it converges to~$\Omega_\infty$ at a higher rate than the convergence of the~$(v_i)_{1\leqslant i\leqslant n}$. More precisely, we will prove the following theorem:

\begin{theorem}\label{thm-ode}
Given~$N$ positive real numbers~$(m_i)_{1\leqslant i\leqslant N}$ with~$\sum_{i=1}^Nm_i=1$, and~$N$ unit vectors~$v_i^0\in\mathbb{S}$ (for~$1\leqslant i\leqslant N$) such that~$v_i^0\neq v_j^0$ for all~$i\neq j$, let~$(v_i)_{1\leqslant i\leqslant N}$ be the solution of the following system of ordinary differential equations :
\begin{equation}
  \label{ode-with-mi}
\frac{\d v_i}{\d t}=P_{v_i^\perp}J, \text{ with } J(t)=\sum_{i=1}^Nm_iv_i(t),
\end{equation}
with the initial conditions~$v_i(0)=v_i^0$ for~$1\leqslant i\leqslant N$, and where~$P_{v_i^\perp}$ denotes the projection on the orthogonal of~$v_i$.

If~$J(0)\neq0$, then~$t\mapsto|J(t)|$ is nondecreasing, so~$\Omega(t)=\frac{J(t)}{|J(t)|}\in\mathbb{S}$ is well-defined for all times~$t\geqslant0$. Furthermore there exists~$\Omega_\infty\in\mathbb{S}$ such that~$\Omega(t)$ converges to~$\Omega_\infty(t)$ as~$t\to+\infty$, and there are only two types of possible asymptotic regimes, which are described below.
\begin{itemize}
\item[(i)] All the vectors~$v_i$ are converging to~$\Omega_\infty$. Then this convergence occurs at an exponential rate~$1$, and~$\Omega$ is converging to~$\Omega_\infty$ at an exponential rate~$3$. More precisely, there exists~$a_i\in\{\Omega_\infty\}^\perp\subset\mathbb{R}^n$, for~$1\leqslant i\leqslant N$ such that~$\sum_{i=1}^Nm_ia_i=0$ and that, as~$t\to+\infty$,
\begin{align*}
v_i(t)&=(1-|a_i|^2e^{-2t})\Omega_\infty+e^{-t}a_i +O(e^{-3t})\quad \text{for }1\leqslant i\leqslant N,\\
\Omega(t)&=\Omega_\infty+O(e^{-3t}).
\end{align*}
\item[(ii)] There exists~$i_0$ such that~$v_{i_0}$ converges to~$-\Omega_\infty$. Then~$m_{i_0}<\frac12$, and if we denote~$\lambda=1-2m_{i_0}$, the vector~$v_{i_0}$ converges to~$-\Omega_\infty$ at an exponential rate~$3\lambda$. Furthermore, all the other vectors~$v_i$ for~$i\neq i_0$ converge to~$\Omega_\infty$ at a rate~$\lambda$, and the vector~$\Omega$ converges to~$\Omega_\infty$ at a rate~$3\lambda$. More precisely, there exists~$a_i\in\{\Omega_\infty\}^\perp\subset\mathbb{R}^n$, for~$i\neq i_0$ such that~$\sum_{i\neq i_0}m_ia_i=0$ and that, as~$t\to+\infty$,
\begin{align*}
v_i(t)&=(1-|a_i|^2e^{-2\lambda t})\Omega_\infty+e^{-\lambda t}a_i +O(e^{-3\lambda t})\quad \text{for }i\neq i_0,\\
v_{i_0}(t)&=-\Omega_\infty+O(e^{-3\lambda t}),\\
\Omega(t)&=\Omega_\infty+O(e^{-3\lambda t}).
\end{align*}
\end{itemize}
\end{theorem}
Notice that the original system~\eqref{ode-intro} can be put as~\eqref{ode-with-mi} with~$m_i=\frac1N$, but the assumption~$v_i^0\neq v_j^0$ for~$i\neq j$ may not be satisfied. Up to renumbering particles and grouping those starting in the same position by setting~$m_i=\frac{k}N$ where~$k$ is the number of particles sharing the same initial condition, we can always fall into the framework of~\eqref{ode-with-mi} with distinct initial conditions. We can finally remark that this system~\eqref{ode-with-mi} is still a particular case of the kinetic equation~\eqref{pde-intro} for a measure given by~$f=\sum_{i=1}^Nm_i\delta_{v_i}$ (see once again Remark~\ref{remark-ode-pde}).

Let us conclude this introduction by saying that these models have also been introduced and studied in different contexts from the one of self-propelled particles. Alignment on the sphere has been introduced as a model of opinion formation in~\cite{aydogdu2017opinion,caponigro2015nonlinear}. The kinetic equation~\eqref{pde-intro} with a diffusion term corresponds to the evolution of rodlike polymers with dipolar potential~\cite{fatkullin2005critical}. Finally the two-dimensional case, where~$\mathbb{S}$ is the unit circle, can correspond to the evolution of identical Kuramoto oscillators. The results we present here were first exposed in detail (with the same proofs as in the present paper) by the first author in the CIMPA Summer School “Mathematical Modeling in Biology and Medicine” in June 2016. They are somewhat similar to those of~\cite{benedetto2015complete} in dimension two, in the context of Kuramoto oscillators, a work that has been raised to us during the presentation of Bastien Fernandez in the workshop “Life Sciences” of the trimester “Stochastic Dynamics out of equilibrium” in May 2017. Very recently, a work~\cite{ha2018relaxation} on generalization of Kuramoto oscillators in higher dimensions, the so-called Lohe oscillators, recovers the same kind of results, although not using exactly the same techniques and not obtaining the precise estimates of Theorem~\ref{thm-ode}. The estimates given by Proposition~\eqref{prop-rates-regular} are also new, as far as we know.

This paper is divided in two main parts. After this introduction, Section~\ref{section-pde} is devoted to the kinetic equation~\eqref{pde-intro}. It is divided in two subsections, the first one being dedicated to the proof of Theorem~\ref{thm-pde}, and the second one giving more precise estimates of convergence in case of symmetries in the initial condition. Section~\ref{section-ode} concerns the system of differential equations~\eqref{ode-with-mi} and the proof of Theorem~\ref{thm-ode}. Even if some conclusions can be drawn using Theorem~\ref{thm-pde} thanks to~Remark~\ref{remark-ode-pde}, we try to make the two parts independent and the proofs self-contained, so the reader interested in Theorem~\ref{thm-ode} can directly jump to this last section.

\section{The continuum model}\label{section-pde}
\subsection{Proof of Theorem~\ref{thm-pde}}
We start with a proposition about well-posedness of the kinetic equation~\eqref{pde-intro}. We proceed for instance as in~\cite{spohn1991large}. We denote by~$\mathcal{P}(\mathbb{S})$ the set of probability measures on~$\mathbb{S}$. In this set we consider the Wasserstein distance~$W_1$ (also called bounded Lipschitz distance) given by~$W_1(\mu,\nu)=\inf_{\varphi\in\mathrm{Lip}_1(\mathbb{S})}|\int_{\mathbb{S}}\varphi\,\d \mu-\int_\mathbb{S}\varphi\,\d\nu|$ for~$\mu$ and~$\nu$ in~$\mathcal{P}(\mathbb{S})$, where~$\mathrm{Lip}_1$ is the set of functions~$\varphi$ such that for all~$u,v$ in~$\mathbb{S}$, we have~$|\varphi(u)-\varphi(v)|\leqslant|v-u|$. This distance corresponds to the weak convergence of probability measures :~$W_1(\mu_n,\mu)\to0$ if and only if for any continuous function~$\varphi:\mathbb{S}\to\mathbb{R}$, we have~$\int_\mathbb{S}\varphi\, \d \mu_n\to\int_\mathbb{S}\varphi\,\d \mu$. The well-posedness result is stated in the space~$C(\mathbb{R}_+,\mathcal{P}(\mathbb{S}))$ of family of probability measures weakly continuous with respect to time:
\begin{proposition} Given~$T>0$ and~$f_0\in\mathcal{P}(\mathbb{S})$, there exists a unique weak solution~$f\in C([0,T],\mathcal{P}(\mathbb{S}))$ to the equation~\eqref{pde-intro} with initial condition~$f_0$, in the sense that for all~$t\in[0,T]$, and for all~$\varphi\in C^1(\mathbb{S})$, we have 
  \begin{equation}
    \frac{\d}{\d t}\int_\mathbb{S}\varphi(v)f(t,v)\,\d v=\int_\mathbb{S}J_{f(t,\cdot)}\cdot\nabla_v\varphi(v)f(t,v)\,\d v,\label{pde-weak}\end{equation}
were we use the notation~$f(t,v)\,\d v$ even if~$f(t,\cdot)$ is not absolutely continuous with respect to the Lebesgue measure on~$\mathbb{S}$, and~$J_{f(t,\cdot)}=\int_\mathbb{S}vf(t,v)\,\d v$.
\end{proposition}
\begin{proof} Notice that the term~$P_{v^\perp}J_f\cdot\nabla_v\varphi$ that we obtain when doing a formal integration by parts of~\eqref{pde-intro} against a test function~$\varphi$ is replaced by~$J_f\cdot\nabla_v\varphi$ in the weak formulation~\eqref{pde-weak}, since the gradient on the sphere at a point~$v$ is already orthogonal to~$v$. The proof of this proposition relies on the fact that the linear equation corresponding to~\eqref{pde-intro} when replacing~$J_f$ by an external given “alignment field”~$\mathcal{J}\in C(\mathbb{R}_+,\mathbb{R}^n)$ is also well-posed. Indeed the solution to this linear equation, namely
  \begin{equation}\partial_tf+\nabla_v\cdot(P_{v^\perp}\mathcal{J}(t)f)=0\quad \text{with}\quad f(0,\cdot)=f_0,\label{pde-linear}
  \end{equation}
  is given by the image measure of~$f_0$ by the flow~$\Phi_t$ of the differential equation~$\frac{\d v}{\d t}=P_{v^\perp}\mathcal{J}(t)$. In detail, if~$\Phi_t$ is the solution of
  \begin{equation}
    \begin{cases}\frac{\d \Phi_t}{\d t}=P_{\Phi_t^\perp}\mathcal{J}(t),\\\Phi_0(v)=v,
    \end{cases}\label{ode-flow}
  \end{equation}
  then the solution~$f(t,\cdot)=\Phi_t\#f_0$ is characterized by the fact that
  \begin{equation}\label{pushforward}
   \forall\varphi\in C(\mathbb{S}), \int_\mathbb{S}\varphi(v)f(t,v)\,\d v=\int_\mathbb{S}\varphi(\Phi_t(v))f_0(v)\,\d v.
  \end{equation}
  Since the differential equation~\eqref{ode-flow} satisfies the assumptions for which the Cauchy-Lipschitz theorem applies, it is well-known (see for instance~\cite{ambrosio2008existence}) that the solution of~\eqref{pde-linear} is unique and given by~$\Phi_t\#f_0$.

  Therefore, if, given~$\mathcal{J}\in C([0,T],\mathbb{R}^n)$, we denote by~$\Psi(\mathcal{J})$ the solution of the linear equation~\eqref{pde-linear}, solving the nonlinear kinetic equation~\eqref{pde-intro} corresponds to finding a fixed point of the map~$f\in C([0,T],\mathcal{P(S)})\mapsto\Psi(J_f)$, or equivalently of the map~$\mathcal{J}\in C([0,T],B)\mapsto J_{\Psi(\mathcal{J})}$, where~$B$ is the closed unit ball of~$\mathbb{R}^n$ (recall that if~$f\in\mathcal{P}(\mathbb{S})$, then~$|J_f|\leqslant1$). The space~$E=C([0,T],B)$ is a complete metric space if the distance is given by~$d_T(\mathcal{J},\bar{\mathcal{J}})=\sup_{t\in[0,T]}|\mathcal{J}(t)-\bar{\mathcal{J}}(t)|e^{-\beta t}$, for an arbitrary~$\beta>0$. Using the fact that~$|(P_{v^\perp}-P_{\bar{v}^\perp})u|\leqslant2|v-\bar{v}|$ if~$|u|\leqslant1$, by a simple Grönwall estimate, if~$\mathcal{J},\bar{\mathcal{J}}\in E$ and~$\Phi_t,\bar{\Phi_t}$ are the associated flow given by~\eqref{ode-flow}, we obtain
  \[|\Phi_t-\bar{\Phi}_t|\leqslant\int_0^t|\mathcal{J}(s)-\bar{\mathcal{J}}(s)|e^{2(t-s)}\d s.\]
  Finally, we get (using the notation~$J_{f}(t)=J_{f(t,\cdot)}$)
 \[\begin{split}|J_{\Psi(\mathcal{J})}(t)-J_{\Psi(\bar{\mathcal{J}})}(t)|&=\left|\int_\mathbb{S} v\, \Psi(\mathcal{J})(t,v)\,\d v -\int_\mathbb{S} v\, \Psi(\bar{\mathcal{J}})(t,v)\,\d v\right|\\&=\left|\int_\mathbb{S}[\Phi_t(v)-\bar{\Phi}_t(v)]f_0(v)\, \d v\right|\\
      &\leqslant\int_0^t|\mathcal{J}(s)-\bar{\mathcal{J}}(s)|e^{2(t-s)}\d s\leqslant d_t(\mathcal{J},\bar{\mathcal{J}})\int_0^te^{2(t-s)+\beta s}\d s.
    \end{split} \]
  Therefore when~$\beta>2$ we get~$|J_{\Psi(\mathcal{J})}(t)-J_{\Psi(\bar{\mathcal{J}})}(t)|e^{-\beta t}\leqslant\frac1{\beta-2}d_t(\mathcal{J},\bar{\mathcal{J}})$, so if we take~$\beta>3$, we get that the map~$\mathcal{J}\mapsto J_{\Psi(\mathcal{J})}$ is indeed a contraction mapping from~$E$ to~$E$, which gives the existence and uniqueness of the fixed point.
  \qed
\end{proof}

\begin{remark}\label{remark-sobolev} The well-posedness of the kinetic equation~\eqref{pde-intro} can also be established in Sobolev spaces, by means of harmonic analysis on the sphere and standard Galerkin method (see~\cite{frouvelle2012dynamics}).
\end{remark}

\begin{remark}\label{remark-ode-pde} Using the weak formulation~\eqref{pde-weak} and the definition of the pushforward measure~\eqref{pushforward}, it is possible to show that a convex combination of Dirac masses, of the form~$f(t,\cdot)=\sum_{i=1}^Nm_i\delta_{v_i}(t)$ with~$m_i\geqslant0$ for~$1\leqslant i\leqslant N$ and~$\sum_{i=1}^Nm_i=1$ is a weak solution of~\eqref{pde-intro} if and only if the~$(v_i)_{1\leqslant i\leqslant N}$ are solutions of the system of differential equations~\eqref{ode-with-mi}.
\end{remark}

We are now ready to prove some qualitative properties of the solution to the kinetic equation~\eqref{pde-intro}. Without further notice, we will denote by~$f$ this solution, and by~$\Phi_t$ the flow~\eqref{ode-flow} associated to~$\mathcal{J}=J_f$. The first property is a simple lemma related to the monotonicity of~$|J_f|$.
\begin{lemma}\label{lem-increasing} If~$f$ is a solution of~\eqref{pde-intro}, then~$|J_f|$ is nondecreasing in time. Therefore if~$J_{f_0}\neq0$, the “average orientation”~$\Omega(t)=\frac{J_f(t)}{|J_f(t)|}$ is well defined and smooth. Furthermore its time derivative~$\dot{\Omega}$ tends to~$0$ as~$t\to\infty$.
\end{lemma}
\begin{proof} Notice that if~$J_{f_0}=0$, then~$f(t,\cdot)=f_0$ for all~$t$. To compute the evolution of~$J_f$, we use~\eqref{pde-weak} with~$\varphi(v)=v\cdot e$ for an arbitrary vector~$e$ in~$\mathbb{R}^n$. We obtain, using the fact that~$\nabla_v(v\cdot e)=P_{v^\perp}e$:
  \[e\cdot\frac{\d J_f}{\d t}=J_f\cdot\int_\mathbb{S}P_{v^\perp}ef(t,v)\,\d v=e\cdot M_fJ_f,\]
  where~$M_f$ is the matrix given by~$\int_\mathbb{S}P_{v^\perp}f(t,v)\,\d v$ (it is a symmetric matrix with eigenvalues in~$[0,1]$, as convex combination of orthogonal projections). Since~$M_f$ is continuous in time, then~$J_f$ is~$C^1$, and by the same procedure we can compute the evolution of~$M_f$, which will depend on higher moments of~$f$, to get that~$J_f$ is smooth. More precisely, since any moment is uniformly bounded (the sphere is compact and~$f(t,\cdot)$ is a probability density for all~$t$), we get that all derivatives of~$J_f$ are uniformly bounded in time. Since
  \[\frac12\frac{\d |J_f|^2}{\d t}=J_f\cdot M_fJ_f=\int_\mathbb{S}[|J_f|^2-(v\cdot J_f)^2]f(t,v)\,\d v\geqslant0,\]
  we get the first part of the proposition.

  From now on we suppose that~$J_{f_0}\neq0$, therefore~$\Omega(t)$ is well defined. The function~$\frac12\frac{\d |J_f|^2}{\d t}=|J_f|^2\Omega\cdot M_f\Omega$ being nonnegative, smooth, integrable in~$\mathbb{R}_+$ (since~$|J_f|$ is bounded by~$1$), and with bounded derivative, it is a classical exercise to show that it must converge to~$0$ as~$t\to\infty$. This gives us that~$\Omega\cdot M_f\Omega\to0$ as~$t\to\infty$. Let us now compute the evolution of~$\Omega$. We get
  \begin{equation}\label{omegadot}
    \dot{\Omega}=\frac{1}{|J_f|}\frac{\d J_f}{\d t}-\frac{\d |J_f|}{\d t}\frac{J_f}{|J_f|^2}=M_f\Omega-(\Omega\cdot M_f\Omega)\Omega=P_{\Omega^\perp}(M_f\Omega).
  \end{equation}
  Since~$M_f$ has eigenvalues in~$[0,1]$, we get that~$|M_f\Omega|^2=\Omega\cdot M_f^2\Omega\leqslant\Omega\cdot M_f\Omega$, therefore~$M_f\Omega\to0$ as~$t\to0$. So we get that~$\dot{\Omega}\to0$ as~$t\to\infty$. \qed
\end{proof}
\begin{remark}\label{remark-gradient-flow} The fact that~$|J_f|$ is nondecreasing can be enlightened by the theory of gradient flow in probability spaces~\cite{ambrosio2008gradient}. Indeed, the kinetic equation~\eqref{pde-intro} corresponds to the gradient flow of the functional~$-\frac12|J_f|^2$ for the Wasserstein distance~$W_2$. Therefore the evolution amounts to minimizing in time this quantity. We also remark that since~$|J_f|$ is nondecreasing, by an appropriate change of time, we can recover the equation~$\partial_tf+\nabla_v\cdot(f P_{v^\perp}\Omega)$ which corresponds to the spatial homogeneous version of~\cite{degond2008continuum} without noise. This equation can also be interpreted as a gradient flow~\cite{figalli2018global}.
\end{remark}
The fact that~$\dot{\Omega}\to0$ is not sufficient to prove that~$\Omega$ converges to some~$\Omega_\infty$, we would need~$\dot{\Omega}\in L^1(\mathbb{R}_+)$ and we only have up to now~$\dot{\Omega}\in L^2(\mathbb{R}_+)$ (since we have seen in the proof of Lemma~\ref{lem-increasing} that~$|J_f|^2\Omega\cdot M_f\Omega$ is integrable in time). To fill this gap, one solution is to compute the second derivative of~$\Omega$, and more precisely, to obtain an estimate on~$|\dot{\Omega}|$ corresponding to the assumption of the following lemma, which mainly says that if~$g$ is integrable, then any bounded solution of the differential equation~$y'=y+g$ has to be integrable.

\begin{lemma}\label{lem-L1} Let~$y:\mathbb{R}_+\to\mathbb{R}$ be a nonnegative function such that~$y^2$ is~$C^1$ and bounded. We suppose that there exists a function~$g\in L^1(\mathbb{R}_+)$ such that for all~$t\in\mathbb{R}$, we have
  \begin{equation}
    \frac12\frac{\d}{\d t}y^2=y^2+y\,g.\label{eq-exploding-ode-L1}
  \end{equation}
  Then~$y\in L^1(\mathbb{R}_+)$.
\end{lemma}
\begin{proof}
  Let~$t\geqslant0$ such that~$y(t)>0$. We set~$T=\sup\{s\geqslant t, y>0\text{ on }[t,s]\}$ (we may have~$T=+\infty$).

  We have that~$y$ is~$C^1$, positive and bounded on~$[t,T)$, and satisfies the differential equation~$y'=y+g$, therefore by Duhamel’s formula we have, for~$s\in[t,T)$:
  \[y(s)e^{-s}-y(t)e^{-t}=\int_t^sg(u)e^{-u}\d u.\]
  Letting~$s=T$ (resp.~$s\to+\infty$ if~$T=+\infty$), since~$y(T)=0$ (resp.~$y$ is bounded), we obtain
  \[y(t)=-\int_t^Tg(u)e^{t-u}\d u\leqslant\int_t^\infty|g(u)|e^{t-u}\d u.\]
  This equality being true for any~$t\in\mathbb{R}_+$ (even if~$y(t)=0$), we have by Fubini’s theorem that
  \[\int_0^\infty y(t)\d t\leqslant\int_0^\infty\int_t^\infty|g(u)|e^{t-u}\d u\, \d t=\int_0^\infty|g(u)|(1-e^{-u})\d u,\]
  which is finite by integrability of~$g$.
  \qed
\end{proof}
We are now ready to prove the convergence of~$\Omega$.
\begin{proposition}\label{prop-omega-converges}If~$J_{f_0}\neq0$, then~$\dot{\Omega}\in L^1(\mathbb{R}_+)$, and therefore there exists~$\Omega_\infty\in\mathbb{S}$ such that~$\Omega\to\Omega_\infty$ as~$t\to\infty$.
\end{proposition}
\begin{proof}
  We first compute the derivative of~$M_f$. For convenience, we use the notation~$\langle\varphi(v)\rangle_f$ for~$\int_\mathbb{S}\varphi(v) f(t,v)\,\d v$. Therefore we have~$J_f=\langle v\rangle_f$ and~$M_f=\langle P_{v^\perp}\rangle_f$, and the weak formulation~\eqref{pde-weak} reads
  \begin{equation*}\frac{\d}{\d t}\langle\varphi(v)\rangle_f=J_f\cdot\langle\nabla_v\varphi(v)\rangle_f.\label{pde-weak-short}
  \end{equation*}
  We have, for fixed~$e_1,e_2\in\mathbb{R}^n$:
  \[e_1\cdot M_fe_2=\langle e_1\cdot P_{v^\perp}e_2\rangle_f=e_1\cdot e_2-\langle(e_1\cdot v)(e_2\cdot v)\rangle_f.\]
  Therefore, since~$\nabla_v(e\cdot v)=P_{v^\perp}e$, we obtain
  \begin{align*}
    \frac{\d}{\d t}(e_1\cdot M_fe_2)&=-J_f\cdot\langle(e_2\cdot v)P_{v^\perp}e_1+(e_1\cdot v)P_{v^\perp}e_2\rangle_f\\&=e_1\cdot[-\langle(e_2\cdot v)P_{v^\perp}J_f\rangle_f+\langle J_f\cdot P_{v^\perp}e_2 \, v\rangle_f],
  \end{align*}
  so the term in between the brackets is the derivative of~$M_fe_2$. We then get
  \begin{align}
    \frac{\d}{\d t}(M_f\Omega)&=M_f\dot{\Omega}-|J_f|\langle(\Omega\cdot v)P_{v^\perp}\Omega\rangle_f-|J_f|\langle\Omega\cdot P_{v^\perp}\Omega\, v\rangle_f\nonumber\\
                       &=M_f\dot{\Omega}+2|J_f|\langle(\Omega\cdot v)^2v\rangle_f-|J_f|[\langle(\Omega\cdot v)\Omega+v\rangle_f]\nonumber\\
    &=M_f\dot{\Omega}+2|J_f|\langle(\Omega\cdot v)^2v\rangle_f-2|J_f|^2\Omega.\label{ddtMOmega}
  \end{align}
  Thanks to~\eqref{omegadot}, we finally have
  \begin{align*}\frac{\d}{\d t}\dot{\Omega}&=\frac{\d}{\d t}(M_f\Omega)-(\Omega\cdot M_f\Omega)\dot{\Omega}-(\dot{\Omega}\cdot M_f\Omega)\Omega-\Omega\cdot\frac{\d}{\d t}(M_f\Omega)\, \Omega\\
    &=P_{\Omega^\perp}\frac{\d}{\d t}(M_f\Omega)-(\Omega\cdot M_f\Omega)\dot{\Omega}-(\dot{\Omega}\cdot M_f\Omega)\Omega.
  \end{align*}
  Since~$\Omega$ and~$\dot{\Omega}$ are orthogonal, we have some simplifications by taking the dot product with~$\dot{\Omega}$ and using~\eqref{ddtMOmega}:
  \begin{align}
    \dot{\Omega}\cdot\frac{\d}{\d t}\dot{\Omega}&=\dot{\Omega}\cdot\frac{\d}{\d t}(M_f\Omega)-(\Omega\cdot M_f\Omega)|\dot{\Omega}|^2.\nonumber\\
                                  &=\dot{\Omega}\cdot M_f\dot{\Omega}-2|J_f|[\langle(\Omega\cdot v)^2\, \dot{\Omega}\cdot v\rangle_f]-(\Omega\cdot M_f\Omega)|\dot{\Omega}|^2\nonumber\\
    &=|\dot{\Omega}|^2-\langle(\dot{\Omega}\cdot v)^2\rangle_f-(\Omega\cdot M_f\Omega)|\dot{\Omega}|^2-2|J_f|[\langle(\Omega\cdot v)^2\, \dot{\Omega}\cdot v\rangle_f].\label{ddtOmegadot2}
  \end{align}
  If we define~$u$ to be the unit vector~$\frac{\dot{\Omega}}{|\dot{\Omega}|}$ when~$|\dot{\Omega}|\neq0$ and to be zero if~$|\dot{\Omega}|=0$, and we set
  \begin{equation}\label{defg}
    g(t)=-|\dot{\Omega}|[\langle(u\cdot v)^2\rangle_f+(\Omega\cdot M_f\Omega)]-2|J_f|\langle(\Omega\cdot v)^2)u\cdot v\rangle_f,
  \end{equation}
  we get that the formula~\eqref{ddtOmegadot2} is written under the following form, corresponding to~\eqref{eq-exploding-ode-L1} with~$y=|\dot{\Omega}|$:
  \[\frac12\frac{\d}{\d t}|\dot{\Omega}|^2=|\dot{\Omega}|^2+|\dot{\Omega}|g(t).\]
  Our goal is to show that~$g\in L^1(\mathbb{R}_+)$ in order to apply Lemma~\ref{lem-L1}. Indeed, thanks to~\eqref{omegadot}, we have that~$|\dot{\Omega}|\leqslant1$ (recall that~$M_f$ is a symmetric matrix with eigenvalues in~$[0,1]$), and~$|\dot{\Omega}|^2$ is~$C^1$.

  As was remarked before in the proof of Lemma~\ref{lem-increasing}, the quantity~$|J_f|^2\Omega\cdot M_f\Omega$ is integrable in time, which gives that~$\Omega\cdot M_f\Omega=\langle1-(\Omega\cdot v)^2\rangle_f$ is integrable. Since~$u$ is colinear to~$\dot{\Omega}$, which is orthogonal to~$\Omega$, we have that~$P_{\Omega^\perp}u=u$, and therefore we get (using the fact that~$|u|\leqslant1$, since~$|u|$ is~$1$ or~$0$)
  \[\langle(u\cdot v)^2\rangle_f=\langle(u\cdot P_{\Omega^\perp}v)^2\rangle_f\leqslant\langle|P_{\Omega^\perp}v|^2\rangle_f=\langle1-(\Omega\cdot v)^2\rangle_f.\]
  This gives that the first term in the definition~\eqref{defg} of~$g$ is integrable in time. Finally, since~$u\cdot\Omega=0$, we have that~$\langle u\cdot v\rangle_f=0$, and we get
  \[|\langle(\Omega\cdot v)^2\,u\cdot v\rangle_f|=|\langle(1-(\Omega\cdot v)^2)u\cdot v\rangle_f|\leqslant\langle1-(\Omega\cdot v)^2\rangle_f,\]
since~$1-(\Omega\cdot v)^2\geqslant0$ and~$|u\cdot v|\leqslant1$ for all~$v\in\mathbb{S}$. This gives that the last term in the definition~\eqref{defg} of~$g$ is also integrable in time. In virtue of Lemma~\ref{lem-L1}, we then get that~$|\dot{\Omega}|$ is integrable. Therefore~$\Omega(t)=\Omega(0)+\int_0^t\dot{\Omega}(s)\d s$ converges as~$t\to+\infty$.
  \qed
\end{proof}

In order to control the distance between~$f$ and~$\delta_{\Omega_\infty}$, we now need to understand the properties of the flow of the differential equation~$\frac{\d v}{\d t}=P_{v^\perp} J_f$.
\begin{proposition}\label{prop-vback}Let~$\mathcal{J}$ be a continuous function~$\mathbb{R}_+\to\mathbb{R}^n$ such that~$t\mapsto|\mathcal{J}(t)|$ is positive, bounded and nondecreasing, and~$\Omega(t)=\frac{\mathcal{J}(t)}{|\mathcal{J}(t)|}$ converges to~$\Omega_\infty\in\mathbb{S}$ as~$t\to\infty$.

  Then there exists a unique~$v_{\mathrm{back}}\in\mathbb{S}$ such that the solution of the differential equation~$\frac{\d v}{\d t}=P_{v^\perp} \mathcal{J}$ with initial condition~$v(0)=v_{\mathrm{back}}$ satisfies~$v(t)\to-\Omega_\infty$ as~$t\to+\infty$. Furthermore, for all~$v_0\neq v_{\mathrm{back}}$, the solution of this differential equation with initial condition~$v(0)=v_0$ converges to~$\Omega_\infty$ as~$t\to+\infty$.
\end{proposition}
\begin{proof} The outline of the proof is the following: we first show that any solution satisfies either~$v(t)\to-\Omega_\infty$ or~$v(t)\to\Omega_\infty$, then we construct~$v_{\mathrm{back}}$, and finally we prove that it is unique. We still denote by~$\Phi_t$ the flow of the differential equation~\eqref{ode-flow}.
  
  We first notice that~$|\mathcal{J}(t)|$ converges to some~$\lambda>0$, therefore~$\mathcal{J}(t)$ converges to~$\lambda\Omega_\infty$ as~$t\to\infty$. Therefore the solution of the equation~$\frac{\d v}{\d t}=P_{v^\perp} \mathcal{J}$ with initial condition~$v(0)=v_0$ is also the solution of a differential equation of the form
  \begin{equation}
    \frac{\d v}{\d t}=\lambda P_{v^\perp} \Omega_\infty + r_{v_0}(t),\label{ode-asymp}
  \end{equation}
  where the remainder term~$r_{v_0}(t)$ converges to~$0$ as~$t\to\infty$, uniformly in~$v_0\in\mathbb{S}$. Let us suppose that~$v(t)$ does not converge to~$-\Omega_\infty$ (that is to say~$v(t)\cdot\Omega_\infty$ does not converge to~$-1$), and let us prove that in this case~$v(t)\to\Omega_\infty$. Taking the dot product with~$\Omega_\infty$ in~\eqref{ode-asymp}, we obtain
  \begin{equation}
    \frac{\d}{\d t}(v\cdot\Omega_\infty)=\lambda[1-(v\cdot\Omega_\infty)^2] + \Omega_\infty\cdot r_{v_0}(t),\label{ode-asymp-vOmega}
  \end{equation}
  so we can use a comparison principle with the one-dimensional differential equation~$y'=\lambda(1-y^2)-\varepsilon$. Since~$\lambda(1-y^2)-\varepsilon$ is positive for~$|y|<\sqrt{1-\frac{\varepsilon}{\lambda}}$ and negative for~$|y|>\sqrt{1-\frac{\varepsilon}{\lambda}}$, any solution starting with~$y(t_0)>-\sqrt{1-\frac{\varepsilon}{\lambda}}$ converges to~$\sqrt{1-\frac{\varepsilon}{\lambda}}$ as~$t\to+\infty$. Since~$v(t)\cdot\Omega_\infty$ does not converge to~$-1$, there exists~$\delta>0$ such that~$v(t)\cdot\Omega_\infty>-1+\delta$ for arbitrarily large times~$t$. For any~$\varepsilon>0$ sufficiently small (such that~$-\sqrt{1-\frac{\varepsilon}{\lambda}}<-1+\delta$), there exists~$t_0\geqslant0$ such that~$v(t_0)\cdot\Omega_\infty>-1+\delta$ and~$|\Omega_\infty\cdot r_{v_0}(t)|\leqslant\varepsilon$ for all~$t\geqslant t_0$. By comparison principle, we then get that~$\liminf_{t\to+\infty} v(t)\cdot\Omega_\infty\geqslant\sqrt{1-\frac{\varepsilon}{\lambda}}$. Since this is true for any~$\varepsilon>0$ sufficiently small, we then get that~$v(t)\cdot\Omega_\infty$ converges to~$1$, that is to say~$v(t)\to\Omega_\infty$ as~$t\to+\infty$.

  Let us now prove that if~$v(t)$ converges to~$\Omega_\infty$, then there exists a neighborhood of~$v_0$ such that the convergence to~$\Omega_\infty$ of solutions starting in this neighborhood is uniform in time. This is done thanks to the same comparison principle. We fix~$\delta>0$ and~$\varepsilon>0$ such that~$-1+\delta>-\sqrt{1-\frac{\varepsilon}{\lambda}}$. We take~$t_0\geqslant0$ such that~$v(t_0)\cdot\Omega_\infty>-1+\delta$ and~$|\Omega_\infty\cdot r_{\tilde{v}_0}(t)|\leqslant\varepsilon$ for any~$\tilde{v}_0\in\mathbb{S}$ and~$t\geqslant t_0$. By continuity of the flow of the equation~$\frac{\d v}{\d t}=P_{v^\perp}\mathcal{J}$, there exists a neighborhood~$B$ of~$v_0$ in~$\mathbb{S}$ such that for any~$\tilde{v}_0\in B$, the solution~$\tilde{v}(t)=\Phi_t(\tilde{v_0})$ of this equation with initial condition~$\tilde{v}_0$ satisfies~$\tilde{v}(t_0)\cdot\Omega_\infty>-1+\delta$. We now look at the equation~$y'=\lambda(1-y^2)-\varepsilon$ starting with~$y(t_0)=-1+\delta$, which converges to~$\sqrt{1-\frac{\varepsilon}{\lambda}}>1-\delta$. There exists~$T$ such that~$y(t)\geqslant1-\delta$ for all~$t\geqslant T$. Therefore, by comparison principle with~\eqref{ode-asymp-vOmega} (where~$v_0$ is replaced by~$\tilde{v}_0$), we get that for all~$\tilde{v}_0\in B$, the solution~$\tilde{v}$ satisfies~$\tilde{v}(t)\cdot\Omega_\infty\geqslant1-\delta$ for all~$t\geqslant T$.

  We are now ready to construct~$v_{\mathrm{back}}$. We take~$(t_n)$ a sequence of increasing times such that~$t_n\to+\infty$ and define~$v_{\mathrm{back}}^n$ as the solution at time~$t=0$ of the backwards in time differential equation~$\frac{\d v^n}{\d t}=P_{(v^n)^\perp}\mathcal{J}$ with terminal condition~$v^n(t_n)=-\Omega_\infty$, that is to say~$v_{\mathrm{back}}^n=\Phi_{t_n}^{-1}(-\Omega_\infty)$. Up to extracting a subsequence, we can assume that~$v_{\mathrm{back}}^n$ converges to some~$v_{\mathrm{back}}\in\mathbb{S}$ and we set~$v(t)=\Phi_t(v_{\mathrm{back}})$. By the first part of the proof, we have that either~$v(t)\to\Omega_\infty$ or~$v(t)\to-\Omega_\infty$ as~$t\to+\infty$. The first case is incompatible with the uniform convergence in time. Indeed, in that case, we would have a neighborhood~$B$ of~$v_{\mathrm{back}}$ and a time~$T$ such that for all~$t\geqslant T$ and all~$\tilde{v}\in B$,~$\Phi_t(\tilde{v})\cdot\Omega_\infty\geqslant0$ (by taking~$\delta=1$ in the previous paragraph). Since we can take~$n$ such that~$t_n\geqslant T$ and~$v_{\mathrm{back}}^n\in B$, this is in contradiction with the fact that~$\Phi_{t_n}(v_{\mathrm{back}}^n)=-\Omega_\infty$.

  It remains to prove that~$v_{\mathrm{back}}$ is unique (which implies that~$\Phi_t^{-1}(-\Omega_\infty)$ actually converges to~$v_{\mathrm{back}}$ as~$t\to+\infty$, thanks to the previous paragraph). This is due to a phenomenon of repulsion of two solutions~$v(t)$ and~$\tilde{v}(t)$ when they are close to~$-\Omega(t)$. Indeed, they satisfy
  \[\frac{\d}{\d t}v\cdot\tilde{v}=v\cdot P_{\tilde{v}^\perp}\mathcal{J}+\tilde{v}\cdot P_{v^\perp}\mathcal{J}=\mathcal{J}\cdot(v+\tilde{v})(1-v\cdot\tilde{v}),\]
  which can be written, since~$\|v-\tilde{v}\|^2=2\,(1-v\cdot\tilde{v})$ as
  \begin{equation}
    \frac{\d}{\d t}\|v-\tilde{v}\|^2=\gamma(t)\|v-\tilde{v}\|^2,\label{odevvtilde}
  \end{equation}
  where~$\gamma(t)=-\mathcal{J}(t)\cdot(v(t)+\tilde{v}(t))$. Let us suppose that both~$v(t)=\Phi_t(v_0)$ and~$\tilde{v}(t)=\Phi_t(\tilde{v}_0)$ converge to~$-\Omega_\infty$ as~$t\to+\infty$. Since~$\mathcal{J}(t)\to\lambda\Omega_\infty$ as~$t\to+\infty$, we have~$\gamma(t)\to2\lambda>0$ as~$t\to+\infty$. Therefore the only bounded solution of the linear differential equation~\eqref{odevvtilde} is the constant~$0$, therefore we have~$v=\tilde{v}$, and thus~$v_0=\tilde{v}_0$.
  \qed
\end{proof}

We are now ready to prove the last part of Theorem~\ref{thm-pde}.
\begin{proposition}Let~$v_{\mathrm{back}}$ be given by Proposition~\ref{prop-vback} with~$\mathcal{J}=J_f$ (we suppose~$J_{f_0}\neq0$). We denote by~$m=\int_\mathbb{S}\mathbf{1}_{v=v_{\mathrm{back}}}f_0(v)\d v$ the initial mass of~$\{v_{\mathrm{back}}\}$. Then~$m<\frac12$ and~$W_1(f(t,\cdot),(1-m)\delta_{\Omega_\infty}+m\delta_{-\Omega_\infty})\to0$ as~$t\to+\infty$.
\end{proposition}
\begin{proof}
  We write~$f_\infty=(1-m)\delta_{\Omega_\infty}+m\delta_{-\Omega_\infty}$. Let~$\varphi\in\mathrm{Lip}_1(\mathbb{S})$. We have
  \[\begin{split}\int_\mathbb{S}\varphi(v)f_\infty(v)\,\d v&=m\varphi(-\Omega_\infty)+(1-m)\varphi(\Omega_\infty)\\ &=m\varphi(-\Omega_\infty)+\int_\mathbb{S}\mathbf{1}_{v\neq v_{\mathrm{back}}}\varphi(\Omega_\infty)f_0(v)\,\d v,\end{split}\]
  and~$\int_\mathbb{S}\varphi(v)f(t,v)\,\d v=\int_\mathbb{S}\varphi(\Phi_t(v))f_0(v)\,\d v$ (recall that~$f(t,\cdot)=\Phi_t\#f_0$ is characterized by~\eqref{pushforward}, where~$\Phi_t$, defined in~\eqref{ode-flow} is the flow of the differential equation~$\frac{\d v}{\d t}=P_{v^\perp}\mathcal{J}$). Therefore we get
  \begin{equation}\label{eqphisplitted}\int_\mathbb{S}\varphi(v)f(t,v)\d v=m\varphi(\Phi_t(v_{\mathrm{back}}))+\int_\mathbb{S}\mathbf{1}_{v\neq v_{\mathrm{back}}}\varphi(\Phi_t(v))f_0(v)\,\d v.
  \end{equation}
  We then obtain
  \[\begin{split}\Big|\int_\mathbb{S}&\varphi(v)f(t,v)\,\d v-\int_\mathbb{S}\varphi(v)f_\infty(v)\,\d v|\\
&\leqslant m|\varphi(\Phi_t(v_{\mathrm{back}}))-\varphi(-\Omega_\infty)|+\int_\mathbb{S}\mathbf{1}_{v\neq v_{\mathrm{back}}}|\varphi(\Phi_t(v))-\varphi(\Omega_\infty)|f_0(v)\,\d v\\
      &\leqslant m|\Phi_t(v_{\mathrm{back}})+\Omega_\infty|+\int_\mathbb{S}\mathbf{1}_{v\neq v_{\mathrm{back}}}|\Phi_t(v)-\Omega_\infty|f_0(v)\,\d v, \end{split} \]
since~$\varphi\in\mathrm{Lip}_1(\mathbb{S})$. We finally get \begin{equation}\label{estimateW1}W_1(f(t,\cdot),f_\infty)\leqslant m|\Phi_t(v_{\mathrm{back}})+\Omega_\infty|+\int_\mathbb{S}\mathbf{1}_{v\neq v_{\mathrm{back}}}|\Phi_t(v)-\Omega_\infty|f_0(v)\,\d v.
\end{equation}
Now, by Proposition~\eqref{prop-vback}, as~$t\to+\infty$ we have~$\Phi_t(v)\to\Omega_\infty$ for all~$v\neq v_{\mathrm{back}}$, and~$\Phi_t(v_{\mathrm{back}})\to-\Omega_\infty$. Therefore by the dominated convergence theorem, the estimate~\eqref{estimateW1} gives that~$W_1(f(t,\cdot),f_\infty)\to0$ as~$t\to+\infty$. It remains to prove that~$m>\frac12$, which comes from Proposition~\ref{prop-omega-converges}, which gives that~$\frac{J_f}{|J_f|}\to\Omega_\infty$ as~$t\to+\infty$. Indeed, applying~\eqref{eqphisplitted} with~$\varphi(v)=v$, we get
\[J_f(t)=m\Phi_t(v_{\mathrm{back}})+\int_\mathbb{S}\mathbf{1}_{v\neq v_{\mathrm{back}}}\Phi_t(v)f_0(v)\,\d v,\] which gives by dominated convergence that, as~$t\to+\infty$, we have
\[J_f(t)\to-m\Omega_\infty+\int_\mathbb{S}\mathbf{1}_{v\neq v_{\mathrm{back}}}\Omega_\infty f_0(v)\,\d v=(1-2m)\Omega_\infty.\]
Since~$\frac{J_f(t)}{|J_f(t)|}\to\Omega_\infty$ as~$t\to+\infty$, we get~$1-2m>0$.
\qed
\end{proof}
\subsection{Symmetries and rates of convergence}
This subsection is dedicated to the study of rates of convergence, based on somewhat explicit solutions in the case where~$\Omega$ is constant in time, which is the case when the initial condition has some symmetries.

\begin{proposition} Let~$G$ be a group of orthogonal transformations under which~$f_0$ is invariant (that is to say~$f_0\circ g=f_0$ and all~$g\in G$) and such that the only fixed points on~$\mathbb{S}$ of every element of~$G$ are two opposite unit vectors that we call~$\pm e_n$. Then the solution~$f(t,\cdot)$ of the partial differential equation~\eqref{pde-intro} is also invariant under all elements of~$g$. Furthermore if~$J_{f_0}\neq0$, then~$J_{f}(t)=\alpha(t)e_n$ with~$\alpha$ positive (up to exchanging~$e_n$ and~$-e_n$), and~$\Omega(t)$ is constantly equal to~$e_n$.
\end{proposition}
\begin{proof} The first part of the proposition comes from the fact that~$t\mapsto f(t,\cdot)\circ g$ is also a solution of~\eqref{pde-intro} (which is well-posed) with the same initial condition. Then, we have by invariance that~$gJ_{f(t,\cdot)}=\int_\mathbb{S}gvf_0(v)\,\d v=\int_\mathbb{S}gvf_0(gv)\,\d v=J_{f(t,\cdot)}$, for all~$g\in G$, and therefore~$\Omega(t)$ is a fixed point of every element of~$g$ and must be equal to~$\pm e_n$.\qed
\end{proof}
Let us mention two simple examples of these kind of symmetries : when~$f_0(v)$ only depends on~$v\cdot e_n$ ($G$ is then the set of isometries having~$e_n$ as fixed point), or when~$f(\sin{\theta}w+\cos{\theta}e_n)=f(-\sin{\theta}w+\cos{\theta}e_n)$ ($G$ is reduced to identity and to~$v\mapsto2e_n\cdot v-v$).

Let us now do some preliminary computations in the case where~$\Omega$ is constant in time. We work in an orthogonal base~$(e_1,\dots,e_n)$ of~$\mathbb{R}^n$ for which~$\Omega=e_n$ is the last vector, and we write write~$J_f(t)=\alpha(t)e_n$, with~$t\mapsto\alpha(t)$ positive and nondecreasing. We will use the stereographic projection
\begin{equation}\label{eq-def-s}s:\quad \begin{matrix}\mathbb{S}\setminus\{-e_n\}&\to&\mathbb{R}^{n-1}\\v&\mapsto&s(v)=\frac{1}{1+v\cdot e_n}P_{e_n^\perp}v,
\end{matrix}
\end{equation}
where we identify~$P_{e_n^\perp}v$ with its first~$n-1$ coordinates. This is a diffeomorphism between~$\mathbb{S}\setminus\{-e_n\}$ and~$\mathbb{R}^{n-1}$, and its inverse is given by
\begin{equation}\label{eq-def-p}p:\quad
  \begin{matrix}\mathbb{R}^{n-1}&\to&\mathbb{S}\setminus\{-e_n\}\subset\mathbb{R}^{n-1}\times\mathbb{R}\\z&\mapsto&p(z)=(\frac{2}{1+|z|^2}\,z,\frac{1-|z|^2}{1+|z|^2}).
  \end{matrix}
\end{equation}
If~$\varphi$ is an integrable function on~$\mathbb{S}$, the change of variable for this diffeomorphism reads
\begin{equation}\label{chg-variable}\int_\mathbb{S}\varphi(v)\,\d v=c_n^{-1}\int_{\mathbb{R}^{n-1}}\frac{\varphi(p(z))}{(1+|z|^2)^{n-1}}\,\d z,
\end{equation}
where the normalization constant is~$c_n=\int_{\mathbb{R}^{n-1}}\frac{\d z}{(1+|z|^2)^{n-1}}$.
If~$v$ is a solution to the differential equation~$\frac{\d v}{\d t}=\alpha(t)P_{v^\perp}e_n$ with~$v\neq-e_n$, a simple computation shows that~$z=s(v)$ satisfies the differential equation~$\frac{\d z}{\d t}=-\alpha(t)z$. Therefore, if we write~$\lambda(t)=\int_0^t\alpha(\tau)\,\d \tau$, we have an explicit expression for the solution~$f$ of the aggregation equation~\eqref{pde-linear}: the pushforward formula~\eqref{pushforward} is given, when~$f_0$ has no atom at~$-e_n$, by
\begin{equation}\label{pushforward-p}\forall\varphi\in C(\mathbb{S}), \int_\mathbb{S}\varphi(v)f(t,v)\,\d v=c_n^{-1}\int_{\mathbb{R}^{n-1}}\frac{\varphi(p(ze^{-\lambda(t)}))f_0(p(z))}{(1+|z|^2)^{n-1}}\, \d z.
\end{equation}
In particular, we have 
\begin{align}\nonumber1-\alpha(t)&=1-J_f(t)\cdot e_n=\int_\mathbb{S}(1-v\cdot e_n)f(t,v)\,\d v\\ \label{oneminusalpha}&=c_n^{-1}\int_{\mathbb{R}^{n-1}}\frac{2|z|^2e^{-2\lambda(t)}f_0(p(z))}{(1+|z|^2e^{-2\lambda(t)})(1+|z|^2)^{n-1}}\, \d z. \end{align}

We are now ready to state the first proposition regarding the rate of convergence towards~$\Omega_\infty$: in the framework of Theorem~\ref{thm-pde}, there is no hope to have a rate of convergence of~$f(t,\cdot)$ with respect to the~$W_1$ distance without further assumption on the regularity of~$f_0$, even if it has no atoms (in this case~$f(t,\cdot)\to\delta_{\Omega_\infty}$ as~$t\to+\infty$). More precisely the following proposition gives the construction of a solution decaying arbitrarily slowly to~$\delta_{\Omega_\infty}$, in contrast with results of local stability of Dirac masses for other models of alignment on the sphere~\cite{degond2014local}, for which as long as the initial condition is close enough to~$\delta_{\Omega_\infty}$, the solution converges exponentially fast in Wasserstein distance.

\begin{proposition} \label{prop-no-rate} Given a smooth decreasing function~$t\mapsto g(t)$ converging to~$0$ (slowly) as~$t\mapsto+\infty$, and such that~$g(0)<\frac12$, there exists a probability density function~$f_0$ such that the solution~$f(t,\cdot)$ of~\eqref{pde-intro} converges weakly to~$\delta_{\Omega_\infty}$, but such that~$W_1(f(t,\cdot),\delta_{\Omega_\infty})\geqslant g(t)$ for all~$t\geqslant0$.
\end{proposition}
\begin{proof}
  We will construct~$f_0$ as a function of the form~$f_0(v)=h(|s(v)|)$, where the stereographic projection~$s$ is defined in~\eqref{eq-def-s}.
  Let us prove that the following choice of~$h$ works, for~$\varepsilon>0$ sufficiently small :
  \[h(r)=b_n\tfrac{(1+r^2)^{n-1}}{r^{n-2}}\Big[\tfrac{1-g(0)}{\varepsilon}\,\mathbf{1}_{0<r<\varepsilon}-\tfrac{g'(\ln r)}{r}\,\mathbf{1}_{r\geqslant1}\Big],\]
  where the normalization constant is~$b_n=\int_{\mathbb{R}_+}\frac{r^{n-2}\,\d r}{(1+r^2)^{n-1}}$. First of all,~$f_0$ is a probability density, since we have, thanks to~\eqref{chg-variable}
  \[\begin{split}\int_\mathbb{S}f_0(v) \d v&=\frac{\int_{\mathbb{R}^{n-1}}\frac{h(|z|)\,\d z}{(1+|z|^2)^{n-1}}}{\int_{\mathbb{R}^{n-1}}\frac{\d z}{(1+|z|^2)^{n-1}}}=\frac{\int_0^{+\infty}\frac{h(r)r^{n-2}\,\d r}{(1+r^2)^{n-1}}}{\int_0^{+\infty}\frac{r^{n-2}\,\d r}{(1+r^2)^{n-1}}}=b_n^{-1}\int_0^{+\infty}\frac{h(r)r^{n-2}\,\d r}{(1+r^2)^{n-1}}\\
      &=\int_0^\varepsilon\tfrac{1-g(0)}{\varepsilon}\,\d r-\int_1^{+\infty}\tfrac{g'(\ln r)}{r}\,\d r=1-g(0)-[g(\ln r)]_1^{+\infty}=1.
    \end{split}\]
  By symmetry, we have that~$J_f(t)=\alpha(t)e_n$. Let us check that~$\alpha(0)>0$. We do as in formula~\eqref{oneminusalpha} :
  \[1-\alpha(t)=b_n^{-1}\int_0^{+\infty}\frac{2r^2e^{-2\lambda(t)}h(r)r^{n-2}\,\d r}{(1+r^2e^{-2\lambda(t)})(1+r^2)^{n-1}}.\]
  We therefore get
  \[\begin{split}1-\alpha(0)&=\int_0^{\varepsilon}\frac{2(1-g(0))r^2 \d r}{(1+r^2)\varepsilon}-\int_1^{+\infty}g'(\ln r)\frac{2r}{1+r^2}\d r\\
      &\leqslant\frac{2\varepsilon^2}3(1-g(0))-2\int_1^\infty g'(\ln r)\frac{\d r}{r}=2g(0)+\frac{2\varepsilon^2}3(1-g(0)),
    \end{split}\]
  which is strictly less than~$1$ as long as~$g(0)<\frac12$ and~$\varepsilon$ is sufficiently small. Therefore in this case we have~$\alpha(0)>0$. This means that~$\Omega(t)=e_n=\Omega_\infty$ for all time~$t$, and thanks to Theorem~\eqref{thm-pde}, since~$f_0$ has no atoms, the solution~$f(t,\cdot)$ converges weakly to~$\delta_{\Omega_\infty}$ as~$t\to+\infty$.

  Let us also remark that~$W_1(f(t,\cdot),\delta_{e_n})=\int_{\mathbb{S}}|v-e_n|f(t,v)\d v$ (see the the proof of the forthcoming Proposition~\ref{prop-rates-regular}), and since we have~$1-v\cdot e_n\leqslant|v-e_n|$, we obtain~$1-\alpha(t)\leqslant W_1(f(t,\cdot),\delta_{e_n})$. Therefore, to prove that the convergence of~$f$ towards~$\delta_{\Omega_\infty}$ is as slow as~$g(t)$, it only remains to prove that~$1-\alpha(t)\geqslant g(t)$. We have~$\lambda(t)\leqslant t$, and so when~$r\geqslant e^{t}$, we get~$re^{-\lambda(t)}\leqslant1$. Since~$x\mapsto\frac{2x}{1+x}$ is increasing, we get~$\frac{2r^2e^{-2\lambda(t)}}{1+r^2e^{-2\lambda(t)}}\geqslant1$. We therefore get
  \[1-\alpha(t)\geqslant-\int_{e^t}^{+\infty}g'(\ln r)\frac{2re^{-2\lambda(t)}}{(1+r^2e^{-2\lambda(t)})}\d r\geqslant-\int_{e^t}^{+\infty}\frac{g'(\ln r)\d r}{r}=g(t),\]
which ends the proof.\qed
\end{proof}

We conclude this subsection by more precise estimates of the rate of convergence in various Wasserstein distances when~$\Omega$ is constant in time and when the initial condition has a density with respect to the Lebesgue measure which is bounded above and below. We write~$a(t)\asymp b(t)$ whenever there exists two positive constants~$c_1,c_2$ such that~$c_1b(t)\leqslant a(t)\leqslant c_2b(t)$ for all~$t\geqslant0$.
We recall the definition of the Wasserstein distance~$W_2$, for two probability measures~$\mu$ and~$\nu$ on~$\mathbb{S}$ :
\[W_2^2(\mu,\nu)=\inf_{\pi}\int_{\mathbb{S}\times\mathbb{S}}|v-w|^2\d\pi(v,w),\]
where the infimum is taken over the probability measures~$\pi$ on~$\mathbb{S}\times\mathbb{S}$ with first and second marginals respectively equal to~$\mu$ and~$\nu$. 
\begin{proposition} \label{prop-rates-regular} Suppose that~$f_0$ has a density with respect to the Lebesgue measure satisfying~$m\leqslant f_0(v)\leqslant M$ for all~$v$ (for some~$0<m<M$), with~$J_{f_0}\neq0$ and such that~$\Omega(t)=e_n$ is constant in time. Then we have
  \begin{align*}
    W_1(f(t,\cdot),\delta_{e_n})&\asymp\begin{cases}(1+t)e^{-t}&\text{ if }n=2,\\e^{-t}&\text{ if }n\geqslant3,
    \end{cases}\\
        W_2(f(t,\cdot),\delta_{e_n})&\asymp\begin{cases}e^{-\frac12t}&\text{ if }n=2,\\\sqrt{1+t}\,e^{-t}&\text{ if }n=3,\\e^{-t}&\text{ if }n\geqslant4.
    \end{cases}
  \end{align*}
\end{proposition}
\begin{proof}
  Let us first give explicit formulas for~$W_1(f(t,\cdot),\delta_{e_n})$ and~$W_2(f(t,\cdot),\delta_{e_n})$. If~$\varphi\in\mathrm{Lip}_1(\mathbb{S})$, we have
  \[\left|\int_\mathbb{S}\varphi(v)f(t,v)\,\d v-\varphi(e_n)\right|\leqslant\int_\mathbb{S}|\varphi(v)-\varphi(e_n)|f(t,v)\,\d v\leqslant\int_\mathbb{S}|v-e_n|f(t,v)\,\d v.\]
  Therefore, by taking the supremum, we get~$W_1(f(t,\cdot),\delta_{e_n})\leqslant\int_\mathbb{S}|v-e_n|f(t,v)\,\d v$. Furthermore, by taking~$\varphi(v)=|v-e_n|$, we get that this inequality is an equality. The explicit expression of~$W_2(f(t,\cdot),\delta_{e_n})$ comes from the fact that the only probability measure on~$\mathbb{S}\times\mathbb{S}$ with marginals~$f(t,\cdot)$ and~$\delta_{e_n}$ is the product measure~$\mu\otimes\delta_{v_0}$, and therefore we have~$W_2^2(f(t,\cdot),\delta_{e_n})=\int_\mathbb{S}|v-e_n|^2f(t,v)\,\d v$. Using the fact that~$|v-e_n|^2=2-2v\cdot e_n$ and the definition~\eqref{eq-def-p} of~$p$, we get~$|p(z)-e_n|=\frac{2|z|}{\sqrt{1+|z|^2}}$. Finally, using~\eqref{pushforward-p}, we obtain
  \begin{equation}\label{W1explicit}W_1(f(t,\cdot),\delta_{e_n})= c_n^{-1}\int_{\mathbb{R}^{n-1}}\frac{2|z|e^{-\lambda(t)}f_0(p(z))}{\sqrt{1+|z|^2e^{-2\lambda(t)}}(1+|z|^2)^{n-1}}\, \d z,
  \end{equation}
  and, as in~\eqref{oneminusalpha}:
  \begin{equation}\label{W2explicit}W_2^2(f(t,\cdot),\delta_{e_n})=2(1-\alpha(t))= c_n^{-1}\int_{\mathbb{R}^{n-1}}\frac{4|z|^2e^{-2\lambda(t)}f_0(p(z))\, \d z}{(1+|z|^2e^{-2\lambda(t)})(1+|z|^2)^{n-1}}.
  \end{equation}
  Thanks to the assumptions on~$f_0$, from~\eqref{W1explicit} we immediately get
  \[W_1(f(t,\cdot),\delta_{e_n})\asymp\int_0^{+\infty}\frac{r^{n-1}e^{-\lambda(t)}\,\d r}{\sqrt{1+r^2e^{-2\lambda(t)}}(1+r^2)^{n-1}},\]
  and for~$n\geqslant3$, since~$\lambda(t)\geqslant0$, we get
  \[\begin{split}0<\int_0^{+\infty}\frac{r^{n-1}\,\d r}{\sqrt{1+r^2}(1+r^2)^{n-1}}&\leqslant\int_0^{+\infty}\frac{r^{n-1}\,\d r}{\sqrt{1+r^2e^{-2\lambda(t)}}(1+r^2)^{n-1}}\\ &\leqslant\int_0^{+\infty}\frac{r^{n-1}\,\d r}{(1+r^2)^{n-1}}<+\infty,
    \end{split}
  \]
  which gives~$W_1(f(t,\cdot),\delta_{e_n})\asymp e^{-\lambda(t)}$. For~$n=2$, we have
  \[\begin{split}\int_0^{+\infty}\frac{re^{-\lambda(t)}\,\d r}{\sqrt{1+r^2e^{-2\lambda(t)}}(1+r^2)}&=\left[\tfrac{e^{-\lambda(t)}}{2\sqrt{1-e^{-2\lambda(t)}}}\ln\Big(\tfrac{\sqrt{1+r^2e^{-2\lambda(t)}}-\sqrt{1-e^{-2\lambda(t)}}}{\sqrt{1+r^2e^{-2\lambda(t)}}+\sqrt{1-e^{-2\lambda(t)}}}\Big)\right]_0^{+\infty}\\
      &=\frac{e^{-\lambda(t)}}{2\sqrt{1-e^{-2\lambda(t)}}}\ln\Big(\frac{1+\sqrt{1-e^{-2\lambda(t)}}}{1-\sqrt{1-e^{-2\lambda(t)}}}\Big).
    \end{split}\]
  Since this last expression is equivalent to~$\lambda(t)e^{-\lambda(t)}$ as~$\lambda(t)\to+\infty$ and converges to~$1$ as~$\lambda(t)\to0$, we then get~$W_1(f(t,\cdot),\delta_{e_n})\asymp (1+\lambda(t))e^{-\lambda(t)}$.

  We proceed similarly for the distance~$W_2$. From the assumptions on~$f_0$ and~\eqref{W2explicit} we get
  \[W_2^2(f(t,\cdot),\delta_{e_n})\asymp1-\alpha(t)\asymp\int_0^{+\infty}\frac{r^{n}e^{-2\lambda(t)}\,\d r}{(1+r^2e^{-2\lambda(t)})(1+r^2)^{n-1}}.\]
  By the same argument of integrability, when~$n\geqslant4$, since~$\int_0^{+\infty}\frac{r^n\, \d r}{(1+r^2)^{n-1}}<+\infty$, we obtain~$1-\alpha(t)\asymp e^{-2\lambda(t)}$. For~$n=2$ we have
  \[\begin{split}\int_0^{+\infty}\frac{r^{2}e^{-2\lambda(t)}\,\d r}{(1+r^2e^{-2\lambda(t)})(1+r^2)}&=\left[\tfrac{e^{-\lambda(t)}\tan^{-1}(e^{-\lambda(t)}r)-e^{-2\lambda(t)}\tan^{-1}(r)}{1-e^{-2\lambda(t)}}\right]_0^{+\infty}\\ &=\frac{\pi\, e^{-\lambda(t)}}{2(1+e^{-\lambda(t)})},
    \end{split}
  \]
  which gives~$1-\alpha(t)\asymp e^{-\lambda(t)}$. For~$n=3$ we have
  \[\begin{split}\int_0^{+\infty}\frac{r^{2}e^{-2\lambda(t)}\,\d r}{(1+r^2e^{-2\lambda(t)})(1+r^2)^2}&=\tfrac{e^ {-2\lambda(t)}}{2(1-e^{-2\lambda(t)})^2}\left[\ln\Big(\tfrac{1+r^2}{1+r^2e^{-2\lambda(t)}}\Big)+\tfrac{1-e^{-2\lambda(t)}}{1+r^2}\right]_0^{+\infty}\\ &=\frac{e^ {-2\lambda(t)}}{2(1-e^{-2\lambda(t)})^2}(2\lambda(t)-1+e^{-2\lambda(t)}).
    \end{split}
  \]
  Since this last expression is equivalent to~$\lambda(t)e^{-2\lambda(t)}$ as~$\lambda(t)\to+\infty$ and converges to~$\frac14$ as~$\lambda(t)\to0$, we then get~$1-\alpha(t)\asymp (1+\lambda(t))e^{-2\lambda(t)}$.

 In all dimensions, we have, since~$\lambda(t)=\int_0^t\alpha(\tau)\d \tau\geqslant\alpha(0)t$, that there exists~$C>0$ such that~$1-\alpha(t)\geqslant Ce^{-\alpha(0)t}$. Therefore, integrating in time, we obtain~$t-\lambda(t)\geqslant\widetilde{C}e^{-\alpha(0)t}$. This gives, since~$\lambda(t)\leqslant t$, that~$e^{-\lambda(t)}\sim e^{-t}$ and~$1+\lambda(t)\asymp 1+t$. Combining this with all the estimates we obtain so far (and reminding that~$W_2(f(t,\cdot),\delta_{e_n})\asymp\sqrt{1-\alpha(t)}$ ends the proof.
  \qed
\end{proof}

Interestingly, the estimates given by Proposition~\ref{prop-rates-regular} depend on the dimension and on the chosen distance. We expect that these estimates still hold when~$\Omega$ depends on time, and, as in the result of Theorem~\ref{thm-ode}, we expect to have an even better rate of convergence of~$\Omega$ towards~$\Omega_\infty$.

\section{The particle model}\label{section-ode}
The object of this section is to prove Theorem~\eqref{thm-ode}, and we divide it into several propositions. We take~$N$ positive real numbers~$(m_i)_{1\leqslant i\leqslant N}$ with~$\sum_{i=1}^Nm_i=1$, and~$N$ unit vectors~$v_i^0\in\mathbb{S}$ (for~$1\leqslant i\leqslant N$) such that~$v_i^0\neq v_j^0$ for all~$i\neq j$. We denote by~$(v_i)_{1\leqslant i\leqslant N}$ the solution of the system of differential equation~\eqref{ode-with-mi} :
\[\frac{\d v_i}{\d t}=P_{v_i^\perp}J, \text{ with } J(t)=\sum_{i=1}^Nm_iv_i(t),\]
with the initial conditions~$v_i(0)=v_i^0$ for~$1\leqslant i\leqslant N$.

\begin{proposition}\label{unique-back-II}
  If~$J(0)\neq0$, then~$|J|$ is nondecreasing, so~$\Omega(t)=\frac{J(t)}{|J(t)|}\in\mathbb{S}$ is well-defined for all times~$t\geqslant0$. We have one of the two following possibilities :
  \begin{itemize}
    \item For all~$1\leqslant i\leqslant N$,~$v_i(t)\cdot\Omega(t)\to1$ as~$t\to+\infty$.
    \item There exists~$i_0$ such that~$v_i(t)\cdot\Omega(t)\to-1$ as~$t\to+\infty$, and for all~$i\neq i_0$, we have~$v_i(t)\cdot\Omega(t)\to1$ as~$t\to+\infty$.
    \end{itemize}
    Furthermore, if we denote by~$\lambda>0$ the limit of~$|J(t)|$ as~$t\to+\infty$, we have for all~$i,j$ in the first possibility (resp. for all~$i\neq i_0$,$j\neq i_0$ in the second possibility),~$\|v_i(t)-v_j(t)\|=O(e^{-(\lambda-\varepsilon)t})$ (for any~$\varepsilon>0$).
  \end{proposition}

  \begin{proof}
    Let us see the differential system as a kind of gradient flow of the following interaction energy (this is reminiscent of the gradient flow structure of the kinetic equation~\eqref{pde-intro}, see Remark~\ref{remark-gradient-flow}):
    \[\mathcal{E}=\frac12\sum_{i,j=1}^Nm_im_j\|v_i-v_j\|^2=\sum_{i,j=1}^Nm_im_j(1-v_i\cdot v_j)=1-|J|^2\geqslant0\]
    Indeed, we then get~$\nabla_{v_i}\mathcal{E}=-2\sum_{j=1}^Nm_im_jP_{v_i^\perp}v_j=-2m_iP_{v_i^\perp}J$ (using the formula~$\nabla_v(u\cdot v)=P_{v^\perp}u$).
    We therefore have~$\frac{\d v_i}{\d t}=-\frac1{2m_i}\nabla_{v_i}\mathcal{E}$, and we obtain
    \begin{equation}\label{dJ2dt}\frac{\d |J|^2}{\d t}=-\frac{\d \mathcal{E}}{\d t}=-\sum_{i=1}^N\nabla_{v_i}\mathcal{E}\cdot\frac{\d v_i}{\d t}=2\sum_{i=1}^Nm_i\left|\frac{\d v_i}{\d t}\right|^2\geqslant0.
  \end{equation}
  This gives that~$|J|$ is nondecreasing in time. So we can define~$\Omega(t)=\frac{J(t)}{|J(t)|}$ and
  rewrite~\eqref{dJ2dt} as
  \begin{equation}\label{dJ2dtbis}\frac{\d|J|^2}{\d t}=2\sum_{i=1}^Nm_i|P_{v_i^\perp}J|^2=2|J|^2\sum_{i=1}^Nm_i(1-(v_i\cdot\Omega)^2).
  \end{equation}
  We can compute the time derivative of this quantity and observe that all terms are uniformly bounded in time. Therefore, since it is an integrable function of time (since~$|J|^2\leqslant1$) with bounded derivative, it must converge to~$0$ as~$t\to+\infty$. Therefore we obtain that~$(v_i(t)\cdot\Omega(t))^2\to1$ for all~$1\leqslant i\leqslant N$.
  Let us now take~$1\leqslant i,j\leqslant N$ and estimate~$\|v_i-v_j\|$. We have
  \begin{align}\nonumber\frac12\frac{\d}{\d t}\|v_i-v_j\|^2&=-\frac{\d}{\d t}(v_i\cdot v_j)=-|J|(v_j\cdot P_{v_i^\perp}\Omega+v_i\cdot P_{v_j^\perp}\Omega)\\ \nonumber &=-|J|\,(\Omega\cdot v_i+\Omega\cdot v_j)(1-v_i\cdot v_j)\\ &=-|J|\,\Omega\cdot\frac{v_i+v_j}2\|v_i-v_j\|^2.\label{dvivj}
  \end{align}
  Therefore if~$v_i\cdot\Omega\to1$ and~$v_j\cdot\Omega\to1$, we get~$\frac12\frac{\d}{\d t}\|v_i-v_j\|^2\leqslant-(\lambda-\varepsilon)\|v_i-v_j\|^2$ for~$t$ sufficiently large, and therefore we obtain~$\|v_i-v_j\|^2=O(e^{-2(\lambda-\varepsilon)t})$.

  Finally if~$v_i\cdot\Omega\to-1$ and~$v_j\cdot\Omega\to-1$, for~$t$ sufficiently large (say~$t\geqslant t_0$) we obtain~$\frac12\frac{\d}{\d t}\|v_i-v_j\|^2\geqslant(\lambda-\varepsilon)\|v_i-v_j\|^2$. This is the same phenomenon of repulsion as~\eqref{odevvtilde} in the previous part, and the only bounded solution to this differential inequality is when~$v_i(t_0)=v_j(t_0)$, which means, by uniqueness that~$v_i^0=v_j^0$ and therefore~$i=j$. This means that if there is an index~$i_0$ such that~$v_{i_0}(t)\cdot\Omega(t)\to-1$, then for all~$i\neq i_0$, we have~$v_{i}(t)\cdot\Omega(t)\to1$ as~$t\to\infty$, and this ends the proof.
  \qed
  \end{proof}
  Let us now study the first possibility more precisely.
  \begin{proposition}\label{prop-no-back}
    Suppose that~$v_i(t)\cdot\Omega(t)\to1$ as~$t\to\infty$ for all~$1\leqslant i\leqslant N$. Then there exists~$\Omega_\infty\in\mathbb{S}$ and~$a_i\in\{\Omega_\infty\}^\perp\subset\mathbb{R}^n$, for~$1\leqslant i\leqslant N$ such that~$\sum_{i=1}^Nm_ia_i=0$ and that, as~$t\to+\infty$,
\begin{align*}
v_i(t)&=(1-|a_i|^2e^{-2t})\Omega_\infty+e^{-t}a_i +O(e^{-3t})\quad \text{for }1\leqslant i\leqslant N,\\
\Omega(t)&=\Omega_\infty+O(e^{-3t}).
\end{align*}
\end{proposition}

\begin{proof}
  We first have~$|J(t)|=J(t)\cdot\Omega(t)=\sum_{i}m_iv_i(t)\cdot\Omega(t)\to1$ as~$t\to\infty$. Therefore~$\lambda=1$, and thanks to the estimates of Proposition~\ref{unique-back-II} (first possibility), for all~$i,j$ we have~$1-v_i\cdot v_j=\frac12\|v_i-v_j\|^2=O(e^{-2(1-\varepsilon)t})$. Summing with weights~$m_j$, we obtain~$1-v_i\cdot J=O(e^{-2(1-\varepsilon)t})$. Plugging back this into~\eqref{dvivj}, we obtain
  \[\frac12\frac{\d}{\d t}\|v_i-v_j\|^2=-\big(1+O(e^{-2(1-\varepsilon)t})\big)\|v_i-v_j\|^2.\]
  We therefore obtain~$\|v_i-v_j\|^2=\|v_i^0-v_j^0\|^2e^{-\int_0^t(1+O(e^{-2(1-\varepsilon)\tau}))\d \tau}=O(e^{-2t})$. This is the same estimate as previously without the~$\varepsilon$. Therefore, similarly, we get~$1-v_i\cdot J=O(e^{-2t})$, which gives~$1-|J|^2=O(e^{-2t})$ by summing with weights~$m_i$. We finally obtain~$1-v_i\cdot\Omega=1-v_i\cdot J+(|J|-1)v_i\cdot\Omega=O(e^{-2t})$, therefore~$|P_{v_i^\perp}\Omega|^2=|P_{\Omega^\perp}v_i|^2=1-(v_i\cdot\Omega)^2=O(e^{-2t})$.

  Let us now compute the evolution of~$\Omega$, as in~\eqref{omegadot}. Since~$\frac{\d J}{\d t}=\sum_{i}m_iP_{v_i^\perp}J$, we use~\eqref{dJ2dtbis} to get~$\frac{\d|J|}{\d t}=|J|\sum_{i}m_i|P_{v_i^\perp}\Omega|^2=O(e^{-2t})$, and we obtain
  \begin{align*}\frac{\d \Omega}{\d t}=\frac{1}{|J|}&\frac{\d J}{\d t}-\frac{\d |J|}{\d t}\frac{J}{|J|^2} =\sum_{i}m_iP_{v_i^\perp}\Omega-\sum_{i}m_i|P_{v_i^\perp}\Omega|^2\Omega \\
     &=-\sum_i m_i(v_i\cdot\Omega)(v_i-(v_i\cdot\Omega)\Omega) =-\sum_{i}m_i(v_i\cdot\Omega)P_{\Omega^\perp}v_i.
  \end{align*}
  Since~$\sum_{i}m_iP_{\Omega^\perp}v_i=P_{\Omega^\perp}J=0$, we can then add this quantity to the previous identity to get
  \begin{equation}\label{dOmegadt}
    \frac{\d \Omega}{\d t}=\sum_{i}m_i(1-v_i\cdot\Omega)P_{\Omega^\perp}v_i.
  \end{equation}
We therefore get~$|\frac{\d \Omega}{\d t}|\leqslant\sum_im_i(1-v_i\cdot\Omega)|P_{\Omega^\perp}v_i|=O(e^{-3t})$. Therefore~$\Omega$ converges towards~$\Omega_\infty\in\mathbb{S}$ and we have~$\Omega=\Omega_\infty+O(e^{-3t})$.

  Finally, to get the precise estimates for the~$v_i$, we compute their second derivative.
  \begin{equation}
    \label{d2vidt}\frac{\d^2 v_i}{\d t^2}=\frac{\d}{\d t}P_{v_i^\perp}J=P_{v_i^\perp}\frac{\d J}{\d t}-\frac{\d v_i}{\d t}\cdot J\, v_i-v_i\cdot J\,\frac{\d v_i}{\d t}.
  \end{equation}
  We have~$P_{v_i^\perp}\frac{\d J}{\d t}=\frac{\d |J|}{\d t}P_{v_i^\perp}\Omega+|J|P_{v_i^\perp}\frac{\d \Omega}{\d t}=O(e^{-3t})$, since~$P_{v_i^\perp}\Omega=O(e^{-t})$ and~$\frac{\d |J|}{\d t}=O(e^{-2t})$ thanks  to~\eqref{dJ2dtbis}. Then we notice that~$\frac{\d v_i}{\d t}\cdot J=J\cdot P_{v_i^\perp}J=|\frac{\d v_i}{\d t}|^2$ and that~$v_i\cdot J\,\frac{\d v_i}{\d t}=\frac{\d v_i}{\d t}-(1-v_i\cdot J)P_{v_i^\perp}J=\frac{\d v_i}{\d t}+O(e^{-3t})$. At the end we obtain
  \begin{equation}\label{dvi2dt2}
    \frac{\d^2 v_i}{\d t^2}=-\frac{\d v_i}{\d t}-\Big|\frac{\d v_i}{\d t}\Big|^2\,v_i+O(e^{-3t}).
  \end{equation}
  Considering first that~$|\frac{\d v_i}{\d t}|^2=O(e^{-2t})$, the resolution of this differential equation gives~$\frac{\d v_i}{\d t}=-a_ie^{-t}+O(e^{-2t})$ with~$a_i\in\mathbb{R}^n$. Integrating in time, we therefore obtain~$v_i(t)=\Omega_\infty+a_ie^{-t}+O(e^{-2t})$, (we already know that~$v_i(t)$ converges to~$\Omega_\infty$ since~$v(t)\cdot\Omega(t)\to1$). The fact that~$|v_i(t)|=1$ gives us~$a_i\cdot\Omega_\infty e^{-t}=O(e^{-2t})$ and therefore~$a_i\in\{\Omega_\infty\}^\perp$. Summing all these estimations with weights~$m_i$ and using the fact that~$J-\Omega_\infty=O(e^{-2t})$, we obtain~$\sum_{i}m_ia_i=0$.

  Finally, the more precise estimate for~$v_i(t)$ up to order~$O(e^{-3t})$ given in the proposition is obtained by plugging back~$|\frac{\d v_i}{\d t}|^2v_i=|a_i|^2e^{-2t}\Omega_\infty+O(e^{-3t})$ into~\eqref{dvi2dt2} and solving it again.
  \qed
\end{proof}
Let us finally study the second possibility.
\begin{proposition}\label{prop-one-back}
Suppose there exists~$i_0$ such that~$v_{i_0}(t)\cdot\Omega(t)\to-1$ as~$t\to\infty$. Then we have~$\lambda=1-2m_{i_0}$ (which gives~$m_{i_0}<\frac12$), and there exists~$\Omega_\infty\in\mathbb{S}$ and~$a_i\in\{\Omega_\infty\}^\perp\subset\mathbb{R}^n$ for~$i\neq i_0$ such that~$\sum_{i\neq i_0}m_ia_i=0$ and that, as~$t\to+\infty$,
\begin{align*}
v_i(t)&=(1-|a_i|^2e^{-2\lambda t})\Omega_\infty+e^{-\lambda t}a_i +O(e^{-3\lambda t})\quad \text{for }i\neq i_0,\\
v_{i_0}(t)&=-\Omega_\infty+O(e^{-3\lambda t}),\\
\Omega(t)&=\Omega_\infty+O(e^{-3\lambda t}).
\end{align*}
\end{proposition}
\begin{proof} First of all we have~$|J(t)|=\Omega(t)\cdot J(t)=\sum_im_iv_i(t)\cdot\Omega(t)$ which converges as~$t\to\infty$ towards~$\lambda=\sum_{i\neq i_0}m_i-m_{i_0}=1-2m_{i_0}$. The proof then follows closely the one of Proposition~\ref{prop-no-back}, except for the case of~$v_{i_0}$. Indeed, Proposition~\ref{unique-back-II} only gives estimates on~$\|v_i-v_j\|$ (and therefore on~$v_i\cdot v_j$) when~$i\neq i_0$ and~$j\neq i_0$. To estimate more precisely the quantity~$v_{i_0}\cdot v_i$, let us prove that~$-v_{i_0}$ must be in the convex cone spanned by~$0$ and all the~$v_i$,~$i\neq i_0$. The idea is that a configuration which is in a convex cone stays in it for all time.

  Let us suppose that all the~$v_i$ (including~$i=i_0$) satisfy~$e\cdot v_i(t_0)\geqslant c$ for some~$c>0$,~$t_0\geqslant0$ and~$e\in\mathbb{S}$ (the direction of the cone). We want to prove that~$e\cdot v_i(t)\geqslant c$ for all~$i$ and for all~$t\geqslant t_0$. If not, we denote by~$t_1>t_0$ a time such that~$e\cdot v_i(t)\geqslant0$ for all~$i$ on~$[t_0,t_1]$, but with~$e\cdot v_j(t_1)<c$ for some~$j$. On~$[t_0,t_1]$, we have
  \begin{equation}\label{devidt}\frac{\d (e\cdot v_i)}{\d t}=e\cdot J-(e\cdot v_i)(v_i\cdot J)\geqslant e\cdot J-(e\cdot v_i),
  \end{equation}
  since~$v_i\cdot J\leqslant|J|\leqslant1$ and~$e\cdot v_i\geqslant0$ on~$[t_0,t_1]$. Summing with weights~$m_i$, we obtain~$\frac{\d (e\cdot J)}{\d t}\geqslant0$. Therefore, since~$e\cdot J(t_0)\geqslant c$, we obtain~$e\cdot J(t)\geqslant c$ on~$[t_0,t_1]$, and the estimation~\eqref{devidt} becomes~$\frac{\d(e\cdot v_i)}{\d t}\geqslant c-(e\cdot v_i)$. By comparison principle, this tells us that~$e\cdot v_i\geqslant c$ on~$[t_0,t_1]$ for all~$i$, which is a contradiction.

  Let us now fix~$t_0\geqslant0$. We want to prove that there exists~$\alpha_i\geqslant0$ for~$i\neq i_0$ such that~$-v_{i_0}=\sum_{i\neq i_0}\alpha_iv_i$. Using Farkas' Lemma, it is equivalent to prove that this is not possible to find~$e\in\mathbb{S}$ such that~$e\cdot v_i(t_0)\geqslant0$ for all~$i\neq i_0$ and~$e\cdot(-v_{i_0})<0$. By contradiction, if such a~$e$ exists, we would have~$e\cdot J(t_0)\geqslant m_{i_0}e\cdot v_{i_0}>0$ and for~$i\neq i_0$, as in~\eqref{devidt}, if~$e\cdot v_i(t_0)=0$ we get~$\frac{\d (e\cdot v_i)}{\d t}|_{t=t_0}=e\cdot J(t_0)>0$. Therefore for~$\delta>0$ sufficiently small, we have~$e\cdot v_i(t_0+\delta)>0$ for all~$i$ (including~$i_0$, and those for which~$e\cdot v_i(t_0)>0$). Therefore there exists~$c>0$ such that for all~$i$,~$e\cdot v_i(t_0+\delta)\geqslant c$, and by the previous paragraph, we get that~$e\cdot v_i(t)\geqslant c$ for all~$t\geqslant t_0+\delta$. We therefore get~$e\cdot\Omega(t)\geqslant\frac1{|J(t)|}e\cdot J(t)\geqslant\frac{c}{|J(0)|}$ for all~$t\geqslant t_0+\delta$. Finally, since~$\|v_{i_0}(t)+\Omega(t)\|^2=2(1+v_{i_0}(t)\cdot\Omega(t))\to0$ as~$t\to\infty$, this is in contradiction with the fact that~$e\cdot(v_{i_0}+\Omega(t))\geqslant(1+\frac1{|J|(0)})c>0$ for all~$t\geqslant t_0+\delta$.

  In conclusion we have that for all~$t\geqslant0$, there exists~$\alpha_i(t)\geqslant0$ for~$i\neq i_0$ such that~$-v_{i_0}(t)=\sum_{i\neq i_0}\alpha_i(t)v_i(t)$. We thus obtain, for~$i\neq i_0$
  \begin{equation}\label{vivi0}v_i(t)\cdot v_{i_0}(t)=-\sum_{i\neq i_0}\alpha_i+\sum_{j\neq i_0}\alpha_i(1-v_j(t)\cdot v_i(t))\leqslant-1+O(e^{-2(\lambda-\varepsilon)t}),
  \end{equation}
  since~$1=\|v_{i_0}(t)\|\leqslant\sum_{i\neq i_0}\alpha_i\|v_i(t)\|=\sum_{i\neq i_0}\alpha_i$, and thanks to Proposition~\eqref{unique-back-II}. Since~$v_i(t)\cdot v_{i_0}(t)\geqslant-1$, this gives~$v_i(t)\cdot v_{i_0}(t)=-1+O(e^{-2(\lambda-\varepsilon)t})$. From there, we have, if~$i\neq i_0$,
  \[\begin{split}v_i\cdot J&=\sum_{i\neq i_0}m_j\,v_i\cdot v_j-m_{i_0}v_i\cdot v_{i_0}\\ &=\sum_{i\neq i_0} (m_j+O(e^{-2(\lambda-\varepsilon)t})-m_{i_0}+O(e^{-2(\lambda-\varepsilon)t})=\lambda+O(e^{-2(\lambda-\varepsilon)t}).\end{split}\]
Plugging this into~\eqref{dvivj}, for~$i\neq i_0$ and~$j\neq i_0$, we obtain
   \[\frac12\frac{\d}{\d t}\|v_i-v_j\|^2=-\big(\lambda+O(e^{-2(\lambda-\varepsilon)t})\big)\|v_i-v_j\|^2.\]
   We therefore obtain, as in the proof of Proposition~\eqref{prop-no-back},~$1-v_i\cdot v_j=O(e^{-2\lambda t})$. As in~\eqref{vivi0}, we now get~$v_i\cdot v_{i_0}=-1+O(e^{-2\lambda t})$. Finally, by summing with weights~$m_j$, we obtain~$v_i\cdot J=\lambda+O(e^{-2\lambda t})$ for~$i\neq i_0$ and~$v_{i_0}\cdot J=-\lambda+O(e^{-2\lambda t})$. Therefore, by summing once again with weights~$m_i$, we get~$|J|^2=\lambda^2+O(e^{-2\lambda t})$. This allows to get~$1-v_i\cdot\Omega=O(e^{-2\lambda t})$ and~$|P_{v_i^\perp}\Omega|=O(e^{-\lambda t})$ when~$i\neq i_0$, and~$1+v_{i_0}\cdot\Omega=O(e^{-2\lambda t})$. Unfortunately this is not enough to use~\eqref{dOmegadt} to obtain a decay at rate~$3\lambda$ : we obtain
   \begin{equation}\label{absdOmegadt}\Big|\frac{\d \Omega}{\d t}\Big|\leqslant O(e^{-3\lambda t})+m_{i_0}(1-v_{i_0}\cdot\Omega)|P_{v_{i_0}^\perp}\Omega|.
   \end{equation}
   However, since~$|P_{v_{i_0}^\perp}\Omega|^2=1-(v_{i_0}\cdot\Omega)^2=(1-v_{i_0}\cdot\Omega)(1+v_{i_0}\cdot\Omega)=O(e^{-2\lambda t})$, we obtain at least~$|\frac{\d \Omega}{\d t}\big|\leqslant O(e^{-\lambda t})$, which gives the existence of~$\Omega_\infty\in\mathbb{S}$ such that~$\Omega(t)=\Omega_\infty+O(e^{-\lambda t})$. To get the rate~$3\lambda$, we have to be a little bit more careful, and use the same kind of trick as in Lemma~\ref{lem-L1} of the first part : if we have a differential equation of the form~$y'=y+O(e^{-\beta t})$, and furthermore that~$y$ is bounded, then we must have~$y=O(e^{-\beta t})$. Indeed, by Duhamel’s formula, we get~$y=y_0e^t+O(e^{-\beta t})$ and the only bounded solution corresponds to~$y_0=0$. We apply this to~$y=\frac{\d v_{i_0}}{\d t}$. We have, as in~\eqref{d2vidt}
   \begin{align}\nonumber\frac{\d^2 v_{i_0}}{\d t^2}&=P_{v_{i_0}^\perp}\frac{\d J}{\d t}-\frac{\d v_{i_0}}{\d t}\cdot J\, v_{i_0}-v_{i_0}\cdot J\,\frac{\d v_{i_0}}{\d t}\\
     &=P_{v_{i_0}^\perp}\frac{\d J}{\d t}-\left|\frac{\d v_{i_0}}{\d t}\right|^2\, v_{i_0} + \lambda\,\frac{\d v_{i_0}}{\d t}+O(e^{-3\lambda t}).\label{d2vi0dt}
   \end{align}
   We have
   \[ P_{v_{i_0}^\perp}\frac{\d J}{\d t}=P_{v_{i_0}^\perp}\Big[J-\sum_{i=1}^N m_i (v_i\cdot J)v_i\Big]=(1-\lambda)P_{v_{i_0}^\perp}J+\sum_{i=1}^Nm_i (\lambda-v_i\cdot J)P_{v_{i_0}^\perp}v_i.\]
   The term for~$i=i_0$ in this last sum vanishes and we have~$\lambda-v_i\cdot J=O(e^{-2\lambda t})$ for~$i\neq i_0$, as well as~$|P_{v_{i_0}^\perp}v_i|^2=1-(v_{i_0}\cdot v_i)^2=O(e^{-2\lambda t})$. We therefore obtain~$P_{v_{i_0}^\perp}\frac{\d J}{\d t}=(1-\lambda)P_{v_{i_0}^\perp}J+O(e^{-3\lambda t})$, and writing~$y=P_{v_{i_0}^\perp}J=\frac{\d v_{i_0}}{\d t}$, the formula~\eqref{d2vi0dt} becomes~$y'=y-|y|^2\,v_{i_0}+O(e^{-3\lambda t})$. We of course have that~$y$ is bounded, and we even know that~$y=\frac1{|J|}P_{v_{i_0}^\perp}\Omega=O(e^{-\lambda t})$. We can then apply the result once by replacing~$|y|^2$ with~$O(e^{-2\lambda t})$ to get~$y=O(e^{-2\lambda t})$, and then apply it a second time to obtain~$y=O(e^{-3\lambda t})$. This already provides the result~$v_{i_0}(t)=-\Omega_\infty+O(e^{-3\lambda t})$, and looking back at~\eqref{absdOmegadt}, we get that~$\frac{\d \Omega}{\d t}=O(e^{-3\lambda t})$ and therefore~$\Omega(t)=-\Omega_\infty+O(e^{-3\lambda t})$.

   It remains to prove the more precise estimates for~$v_i$ when~$i\neq i_0$, and this is done exactly as in the proof of Proposition~\eqref{prop-no-back}, from formula~\eqref{d2vidt} to the end of the proof, now we know that~$\frac{\d \Omega}{\d t}=O(e^{-3\lambda t})$. The only difference is that~$v_i\cdot J$ converges to~$\lambda$ instead of~$1$, together with the fact that all rates are multiplied by~$\lambda$. For instance, the main estimate~\eqref{dvi2dt2} becomes
   \[\frac{\d^2 v_i}{\d t^2}=-\lambda\frac{\d v_i}{\d t}-\Big|\frac{\d v_i}{\d t}\Big|^2\,v_i+O(e^{-3\lambda t}),\]
  and the rest of the proof does not change.\qed
  \end{proof}

  \section{Acknowledgments}

The authors want to thank the hospitality of Athanasios Tzavaras and the University of Crete, back in 2012, where this work was done and supported by the EU FP7-REGPOT project ``Archimedes
Center for Modeling, Analysis and Computation''.\\
  A.F. acknowledges support from the EFI project ANR-17-CE40-0030 and the Kibord project ANR-13-BS01-0004 of the French National Research Agency (ANR), from the project Défi S2C3 POSBIO of the interdisciplinary mission of CNRS, and the project SMS co-funded by CNRS and the Royal Society.\\
  J.-G. L. acknowledges support from the National Science Foundation under NSF Research Network Grant no. RNMS11-07444 (KI-Net) and grant DMS-1812573.

\end{document}